\documentclass[12pt]{amsart}

\usepackage{amsmath,amssymb,latexsym,amsthm,newlfont,enumerate,xy,cmap}

\usepackage{graphicx}

\DeclareMathOperator{\mult}{mult}
\DeclareMathOperator{\res}{res}

\DeclareMathOperator{\Int}{int}
\DeclareMathOperator{\relint}{relint}
\DeclareMathOperator{\Supp}{Supp}

\DeclareMathOperator{\Diff}{Diff}
\DeclareMathOperator{\Div}{WDiv}

\DeclareMathOperator{\bDiv}{\mathbf{Div}}

\DeclareMathOperator{\Mob}{Mob}
\DeclareMathOperator{\bMob}{\mathbf{Mob}}

\newcommand{\R}{\mathbb{R}}
\newcommand{\Q}{\mathbb{Q}}
\newcommand{\N}{\mathbb{N}}
\newcommand{\Z}{\mathbb{Z}}

\newcommand{\D}{\mathbf{D}}

\newcommand{\F}{\mathbf{F}}
\newcommand{\M}{\mathbf{M}}
\newcommand{\A}{\mathbf{A}}
\newcommand{\B}{\mathbf{B}}
\newcommand{\K}{\mathbf{K}}

\newcommand{\m}{\mathbf{m}}
\newcommand{\n}{\mathbf{n}}

\newcommand{\OO}{\mathcal{O}}

\newcommand{\dist}{\mathrm{dist}}

\newcommand{\mcal}{\mathcal}

\theoremstyle{plain}
\newtheorem{te}{Theorem}[section]
\newtheorem{lem}[te]{Lemma}

\newtheorem{cor}[te]{Corollary}

\newtheorem{prop}[te]{Proposition}
\newtheorem*{conA}{Conjecture A}
\newtheorem*{conB}{Conjecture B}

\theoremstyle{definition}
\newtheorem{de}[te]{Definition}
\newtheorem{re}[te]{Remark}
\newtheorem{ex}[te]{Example}

\title{On Shokurov-type b-divisorial algebras of higher rank}
\date{2 December 2008}
\author{Vladimir Lazi\'c}
\address{Department of Pure Mathematics and Mathematical Statistics, Uni\-ver\-si\-ty of Cambridge, Wilberforce Road, Cambridge CB3 0WB, UK}
\email{V.Lazic@dpmms.cam.ac.uk}

\begin{document}

\begin{abstract}
The purpose of this paper is to lay the foundations for the theory of higher rank b-divisorial algebras of Shokurov type. We develop techniques to deal
with such objects and propose two natural conjectures regarding Shokurov algebras and adjoint algebras. We confirm these conjectures in the case of
affine curves.
\end{abstract}

\maketitle
\bibliographystyle{amsalpha}

\tableofcontents

\section{Introduction}

Let $X$ be a normal algebraic variety and $k(X)$ the field of rational functions on $X$.
In this paper we study algebras of rational functions of rank $r$, that is subalgebras $R$ of $k(X)[T_1,\dots,T_r]$.
The key example of rank $1$ is the canonical algebra
$$R(X/Z,K_X+\Delta)=\bigoplus_{i\geq0}\pi_*\OO_X(\lfloor i(K_X+\Delta)\rfloor),$$
where $\pi\colon X\rightarrow Z$ is a projective morphism between normal varieties, the pair $(X,\Delta)$ is klt and $K_X+\Delta$
is $\Q$-Cartier.

Algebras of higher rank also appear naturally. The key example here is the Cox ring, that is
$$R(X/Z,D_1,\dots,D_r)=\bigoplus_{m\in\N^r}\pi_*\OO_X\big(\big\lfloor\sum m_iD_i\big\rfloor\big),$$
where $D_i$ are $\Q$-divisors on $X$.

There has been a huge progress recently on finite generation of these algebras. It is shown in \cite{BCHM} that canonical algebras
are finitely generated, and also that adjoint dlt Cox rings are finitely generated (for the precise
statement see \cite[Corollary 1.1.9]{BCHM}). Therefore it is desirable to study general natural conditions that imply finite generation of $R$.

The main new input came from \cite{Sho03}. Shokurov treats the rank $1$ case and introduces {\em boundedness\/} and {\em saturation\/} conditions.
Saturation was the main ingredient to prove the finite generation of the canonical algebra, and it behaves well under restriction (see Lemma
\ref{restriction} below).

The purpose of this note is to show that these techniques can be extended to the higher rank algebras and that one could expect,
perhaps surprisingly, that they work.

Here we make two natural conjectures about higher rank finite generation; for the precise definitions of Shokurov and adjoint algebras see Section
\ref{section} below.

\begin{conA}
Let $(X,\Delta)$ be a relative weak Fano klt pair projective over a normal affine variety $Z$ where $K_X+\Delta$ is $\Q$-Cartier.
Let $\mcal{S}\subset\N^r$ be a finitely generated monoid and let $\m\colon\mcal{S}\rightarrow\bMob(X)$ be a superadditive map such that
the system $\m(\mcal{S})$ is bounded and $\A(X,\Delta)$-saturated. Let $\mcal{C}$ be a rational polyhedral cone in $\Int\mcal{S}_\R$.

Then the Shokurov algebra $R(X,\m(\mcal{C}\cap\mcal{S}))$ is a finitely generated $\OO_Z$-algebra.
\end{conA}

\begin{conB}
Let $\pi\colon X\rightarrow Z$ be a projective morphism between normal varieties, let $\mcal{S}\subset\N^r$ be a finitely generated monoid
and let $\m\colon\mcal{S}\rightarrow\bMob(X)$ be a superadditive map such that the system $\m(\mcal{S})$ is adjoint.
Let $\mcal{C}$ be a rational polyhedral cone in $\Int\mcal{S}_\R$.

Then the adjoint algebra $R(X,\m(\mcal{C}\cap\mcal{S}))$ is a finitely generated $\OO_Z$-algebra.
\end{conB}

The definition of adjoint algebra presented in this paper is still rather tentative and should be taken with caution.
We expect that its definite form will be clear very soon and similar to the current one.

Ideally we would like the conjectures to extend to the whole cone $\mcal{S}_\R$, however this is in general not possible, see Remark \ref{remark}.

Our main contribution here is to show that these conjectures hold if $X$ is an affine curve (Theorem \ref{corollary} below).
I expect the proof in higher dimensions to proceed along the general lines of the proof presented here and to employ the techniques
developed in this paper.

The major new input is to pass from discrete objects (superadditive maps from finitely generated monoids to spaces of integral b-divisors) to
continuous objects (superlinear maps from rational polyhedral cones to spaces of real b-divisors).
Irrationality issues encountered on the way are solved using Diophantine techniques different in flavour from those applied in
\cite{Sho03,Cor07,HM07}.

The behaviour of the restricted volume function, see \cite{ELMNP2}, applied to
the restricted Cox ring gives a good perspective on what is happening in the general case.

Adjoint b-divisorial algebras of rank $1$ appear as restrictions of canonical algebras to codimension $1$ log canonical centres, see \cite{CKL08}.

Algebras of higher rank also appear naturally if one tries to prove finite generation of the canonical ring by induction on the dimension without
the Minimal Model Programme. This paper is the first step in a larger programme of work towards this goal, see \cite{Laz08}.

\subsection{Notation and conventions}

All varieties in this paper are normal over an algebraically closed field $k$ of characteristic zero.
Throughout I use the language of b-divisors; for an accessible introduction see \cite{Cor07}. In particular, let us recall the following
definitions.

\begin{de}
The group of Weil divisors on a variety $X$ is denoted by $\Div(X)$.

A (proper) {\em model\/} over $X$ is a variety $Y$ with a proper birational morphism $f\colon Y\rightarrow X$. Therefore we have the induced
homomorphism $f_*\colon\Div(Y)\rightarrow\Div(X)$. The group
$$\bDiv(X)=\underleftarrow{\lim}\Div(Y),$$
where the limit is taken over all models $Y$ over $X$, is called the {\em group of b-divisors on $X$\/}, and its elements are b-divisors on $X$.
The group of rational (respectively real) b-divisors is denoted by $\bDiv(X)_{\Q}$ (respectively $\bDiv(X)_{\R}$). The {\em trace\/} of a
b-divisor $\D$ on a model $Y$ is denoted by $\D_Y$.

The {\em proper transform\/} b-divisor $\widehat{D}$ of an $\R$-divisor $D$ is given by $\widehat{D}_Y=f_*^{-1}D$ for every model
$f\colon Y\rightarrow X$. I abuse the notation and write $D$ instead of $\widehat{D}$.

The {\em Cartier closure\/} of an $\R$-Cartier divisor $D$ on $X$ is the b-divisor $\overline{D}$ given by $\overline{D}_Y=f^*D$ for every model
$f\colon Y\rightarrow X$.

A b-divisor $\D$ {\em descends\/} to a model $f\colon Y\rightarrow X$ if $\D=\overline{\D_Y}$ (considered as a
b-divisor on $X$ by the isomorphism $f_*\colon\bDiv(Y)\rightarrow\bDiv(X)$).

A b-divisor $\M$ on $X$ is {\em mobile\/} if there is a model $Y$ over $X$ such that $\M_Y$ is free and $\M=\overline{\M_Y}$. Observe that if a mobile
b-divisor $\M$ descends to a model $W\rightarrow X$, then $\M_W$ is free.

The cone of all mobile b-divisors on $X$ is denoted by $\bMob(X)$, and similarly for rational and real mobile b-divisors.
\end{de}

The {\em canonical\/} b-divisor on $X$ is denoted by $\K_X$.

\begin{de}
Let $(X,\Delta)$ be a pair where $K_X+\Delta$ is a $\Q$-Cartier divisor. The {\em discrepancy\/} b-divisor is defined to be
$$\A(X,\Delta):=\K_X-\overline{K_X+\Delta}.$$
\end{de}

We obviously have $\lceil\A(X,\Delta)\rceil\geq0$ if and only if $(X,\Delta)$ is a klt pair.

In this paper I use the adjunction formula with {\em differents\/} as explained in \cite[Chapter 16]{Kol92}.

I denote the sets of non-negative rational and real numbers by $\Q_+$ and $\R_+$ respectively.
If $\mcal{S}=\sum\N e_i$ is a submonoid of $\N^n$, I denote by $\mcal{S}_\R$ its {\em associated cone\/},
that is $\mcal S_\R=\sum\R_+e_i$. Also I denote $\mcal S_\Q=\sum\Q_+e_i$.

A monoid $\mcal{S}\subset\N^n$ is {\em saturated\/} if $\mcal{S}=\mcal S_\R\cap\N^n$.

Let $e_1,\dots,e_n$ be generators of a monoid $\mcal{S}$ and let $\kappa_1,\dots,\kappa_n$ be positive integers.
The submonoid $\mcal{S}'=\sum_{i=1}^n\N \kappa_ie_i$ of $\mcal{S}$ is called a {\em truncation\/} of $\mcal{S}$.

For any point $s$ in a cone $\mcal{C}\subset\R^n$, the set $\R_+s$ is called a {\em ray\/} in $\mcal{C}$.
A ray in $\mcal{C}$ is {\em rational\/} if it contains a rational point, otherwise it is {\em irrational\/}.
Rays $R_1,\dots,R_m$ are said to be linearly independent if $s_i\in R_i$ are linearly independent for any choice of $s_i$.
A $k$-dimensional plane (or a $k$-plane) in $\R^n$ containing the origin is {\em rational\/} if it is spanned by $k$ linearly
independent rational rays, or equivalently if there are linear functions $\ell_i\colon\R^n\rightarrow\R$ for $i=1,\dots,n-k$ with
rational coefficients such that $\mcal{H}=\bigcap_{i=1}^{n-k}\ker(\ell_i)$.

All $k$-planes in this paper contain the origin unless explicitly stated otherwise.

For a cone $\mcal C\subset\R^n$, I denote $\mcal C_\Q=\mcal C\cap\Q^n$.
The {\em dimension\/} of a cone $\mcal{C}=\sum\R_+e_i$ is the dimension of the space $\sum\R e_i$.

In this paper the {\em relative interior\/} of a cone $\mcal{C}=\sum\R_+e_i\subset\R^n$, denoted by $\relint\mcal{C}$, is the topological interior
of $\mcal{C}$ in the space $\sum\R e_i$ union the origin. If $\dim\mcal{C}=n$, we instead call it the {\em interior\/} of $\mcal{C}$ and
denote it by $\Int\mcal{C}$.

All cones considered are convex and strongly convex, that is they do not contain lines.

If $\mcal{C}\subset\R^n$ is a closed cone and $R$ a ray in the boundary of $\mcal{C}$, a {\em tangent hyperplane\/} to $\mcal{C}$ through $R$ is
any hyperplane $T\supset R$ such that $\mcal{C}$ is contained in one of the half-spaces into which $T$ divides $\R^n$.

\begin{de}
A submonoid $\mcal{S}=\sum\N e_i$  of $\N^n$ (respectively a cone $\mcal{C}=\sum\R_+ e_i$ in $\R^n$) is called {\em simplicial\/}
if its generators $e_i$ are linearly independent in $\R^n$, and the $e_i$ form a {\em basis\/} of $\mcal{S}$
(respectively $\mcal{C}$).
\end{de}

Let $\mcal{C}\subset\R^n$ be a polyhedral cone and let $f\colon\mcal{C}\rightarrow\R$ be a function. We say $f$ is
{\em piecewise linear\/} (PL) on $\mcal{C}$ if there is a finite polyhedral decomposition $\mcal{C}=\bigcup\mcal{C}_i$ such that
$f_{|\mcal{C}_i}$ is linear for every $i$. If in addition $\mcal{C}$ and all $\mcal{C}_i$ are rational cones then we say
$f$ is {\em rationally piecewise linear\/} ($\Q$-PL).

Assume furthermore that $f$ is linear on $\mcal{C}$ and $\dim\mcal{C}=n$. The {\em linear extension of $f$ to $\R^n$\/} is the unique linear function
$\ell\colon\R^n\rightarrow\R$ such that $\ell_{|\mcal{C}}=f$.

\subsection{Acknowledgements}

I would like to wholeheartedly thank my supervisor Alessio Corti for encouragement and support. A paragraph here is not enough to express
how much I have benefited from conversations with him and from his ideas.
I would also like to thank Shigefumi Mori, Burt Totaro, Nick Shepherd-Barron and Anne-Sophie Kaloghiros for useful discussions and comments.
I am supported by Trinity College, Cambridge.

\section{Convex geometry}

Firstly we recall a definition.

\begin{de}
Let $\mcal{C}$ be a cone in $\R^n$ and let $\|\cdot\|$ be any norm on $\R^n$. A function $f\colon\mcal{C}\rightarrow\R$ is
\begin{itemize}
\item {\em superlinear\/} if $\lambda f(x)+\mu f(y)\leq f(\lambda x+\mu y)$
for every $x,y\in\mcal{C}$ and every $\lambda,\mu\geq0$;

\item {\em superadditive\/} if $f(x)+f(y)\leq f(x+y)$ for every $x,y\in\mcal{C}$;

\item {\em positively homogeneous\/} if $f(\lambda x)=\lambda f(x)$ for every $x\in\mcal{C}$ and every $\lambda\geq0$;

\item {\em locally Lipschitz\/} if for every point $x\in\Int\mcal{C}$ there are a closed ball $B_x\subset\mcal{C}$ centred at $x$ and a constant
$\lambda_x$ such that $|f(y)-f(z)|\leq\lambda_x\|y-z\|$ for all $y,z\in B_x$.
\end{itemize}
\end{de}

Every locally Lipschitz function is continuous on $\Int\mcal{C}$. Therefore if a function is locally Lipschitz, we say it is
{\em locally Lipschitz continuous\/}.

The following proposition is an easy consequence of the definitions.

\begin{prop}\label{prop}
A function $f$ is superlinear if and only if it is superadditive and positively homogeneous.
\end{prop}

The next result can be found in \cite{HUL93}.

\begin{prop}\label{Lip}
Let $\mcal{C}$ be a cone in $\R^n$ and let $f\colon\mcal{C}\rightarrow\R$ be a concave function.
Then $f$ is locally Lipschitz continuous on the topological interior of $\mcal{C}$ with respect to any norm $\|\cdot\|$ on $\R^n$.

In particular, let $\mcal{C}$ be a rational polyhedral cone and assume a function
$g\colon\mcal{C}_\Q\rightarrow\Q$ is superadditive and satisfies
$g(\lambda x)=\lambda g(x)$ for all $x\in\mcal{C}_\Q$ and all $\lambda\in\Q_+$. Then $g$ extends to a unique superlinear
function on $\mcal{C}$.
\end{prop}
\begin{proof}
Since $f$ is locally Lipschitz if and only if $-f$ is locally Lipschitz, we can assume $f$ is convex. Fix $x=(x_1,\dots,x_n)\in\Int\mcal{C}$, and
let $\Delta=\{(y_1,\dots,y_n)\in\R_+^n:\sum y_i\leq1\}$. It is easy to check that translations of the domain do not affect the result,
so we may assume $x\in\Int\Delta\subset\Int\mcal{C}$.

Firstly let us prove that $f$ is bounded above on $\Delta$. Let $\{e_i\}$ be the standard basis in $\R^n$, $y=(y_1,\dots,y_n)\in\Delta$ and
let $y_0=1-\sum y_i\geq0$. Then
\begin{multline*}
f(y)=f\Big(\sum y_ie_i+y_0\cdot0\Big)\leq\sum y_if(e_i)+y_0f(0)\\
\leq\max\{f(0),f(e_1),\dots,f(e_n)\}=:M.
\end{multline*}

For each $\gamma>0$ denote $B_x(\gamma)=\{z\in\R^n:\|z-x\|\leq\gamma\}$. Choose $\delta$ such that $B_x(2\delta)\subset\Int\Delta$. Again by
translating the domain and composing $f$ with a linear function we may assume that $x=0$ and $f(0)=0$. Then for all $y\in B_0(2\delta)$ we have
$$-f(y)=-f(y)+2f(0)\leq-f(y)+f(y)+f(-y)=f(-y)\leq M,$$
so $|f|\leq M$ on $B_0(2\delta)$.

Fix $u,v\in B_0(\delta)$. Set $\alpha=\|v-u\|/\delta$ and $w=v+\alpha^{-1}(v-u)\in B_0(2\delta)$ so that
$v=\alpha w/(\alpha+1)+u/(\alpha+1)$. Then convexity of $f$ gives
\begin{align*}
f(v)-f(u)&\leq\frac{\alpha}{\alpha+1}f(w)+\frac{1}{\alpha+1}f(u)-f(u)\\
&=\frac{\alpha}{\alpha+1}\big(f(w)-f(u)\big)\leq2M\alpha=\frac{2M}{\delta}\|v-u\|.
\end{align*}
Similarly $f(u)-f(v)\leq2M\|u-v\|/\delta$, giving
$$|f(v)-f(u)|\leq L\|v-u\|$$
for all $u,v\in B_0(\delta)$ and $L=2M/\delta$.

For the second claim, it is enough to apply the proof of the first part of the lemma with respect to the sup-norm $\|\cdot\|_{\infty}$;
observe that $\|\cdot\|_{\infty}$ takes values in $\Q$ on $\mcal{C}_\Q$. Applied to the interior of $\mcal{C}$ and to
the relative interiors of the faces of $\mcal{C}$
shows $g$ is locally Lipschitz, and therefore extends to a unique superlinear function on the whole $\mcal{C}$.
\end{proof}

The following result is classically referred to as Gordan's lemma.

\begin{lem}\label{gordan}
Let $\mcal{S}\subset\N^r$ be a finitely generated monoid and let $\mcal{C}\subset\R^r$ be a rational polyhedral cone. Then the monoid
$\mcal{S}\cap\mcal{C}$ is finitely generated.
\end{lem}
\begin{proof}
Assume first that $\dim\mcal{C}=r$.
Let $\ell_1,\dots,\ell_m$ be linear functions on $\R^r$ with integral coefficients such that $\mcal{C}=\bigcap_{i=1}^m\{z\in\R^r:\ell_i(z)\geq0\}$
and define $\mcal{S}_0=\mcal{S}$ and $\mcal{S}_i=\mcal{S}_{i-1}\cap\{z\in\R^r:\ell_i(z)\geq0\}$ for $i=1,\dots,m$; observe that
$\mcal{S}\cap\mcal{C}=\mcal{S}_m$. Assuming by induction that $\mcal{S}_{i-1}$ is finitely generated, by \cite[Theorem 4.4]{Swa92} we have that
$\mcal{S}_i$ is finitely generated.

Now assume $\dim\mcal{C}<r$ and let $\mcal{H}$ be a rational hyperplane containing $\mcal{C}$. Let $\ell$ be the linear function with rational
coefficients such that $\mcal{H}=\ker(\ell)$. From the first part of the proof applied to the functions $\ell$ and $-\ell$
we have that the monoid $\mcal{S}\cap\mcal{H}$ is finitely generated. Now we proceed by descending induction on $r$.
\end{proof}

The following lemmas will turn out to be indispensable and they show that in the context of our assumptions it is enough to check additivity
(respectively linearity) of the map at one point only.

\begin{lem}\label{linear}
Let $\mcal{S}=\sum_{i=1}^n\N e_i$ be a monoid and let $f\colon\mcal{S}\rightarrow G$ be a superadditive map to a
monoid $G$. Assume that there is a point $s_0=\sum s_ie_i\in\mcal{S}$ with all $s_i>0$ such that
$f(s_0)=\sum s_if(e_i)$ and that $f(\kappa s_0)=\kappa f(s_0)$ for every positive integer $\kappa$.
Then the map $f$ is additive.
\end{lem}
\begin{proof}
For $p=\sum p_ie_i\in\mcal{S}$, let $\kappa_0$ be a big enough positive integer such that $\kappa_0s_i\geq p_i$ for all $i$. Then we have
\begin{align*}
\sum_{i=1}^n\kappa_0s_if(e_i)&=\kappa_0f(s_0)=f(\kappa_0s_0)\geq f(p)+\sum_{i=1}^nf\big((\kappa_0s_i-p_i)e_i\big)\\
&\geq\sum_{i=1}^nf(p_ie_i)+\sum_{i=1}^nf\big((\kappa_0s_i-p_i)e_i\big)\\
&\geq\sum_{i=1}^np_if(e_i)+\sum_{i=1}^n(\kappa_0s_i-p_i)f(e_i)=\sum_{i=1}^n\kappa_0s_if(e_i).
\end{align*}
Therefore all inequalities are equalities and $f(p)=\sum p_if(e_i)$.
\end{proof}

Analogously we can prove a continuous counterpart of the previous result.

\begin{lem}\label{linear2}
Let $\mcal{C}=\sum_{i=1}^n\R_+e_i$ be a cone in $\R^r$ and let $f\colon\mcal{C}\rightarrow V$ be a superlinear map to a
cone $V$. Assume that there is a point $s_0=\sum s_ie_i\in\mcal{C}$ with all $s_i>0$ such that
$f(s_0)=\sum s_if(e_i)$. Then the map $f$ is linear.
\end{lem}

\section{Forcing Diophantine approximation}

In this section I will prove the following.

\begin{te}\label{PLinitial}
Let $\mcal{S}\subset\N^r$ be a finitely generated monoid and let $f\colon\mcal{S}_\R\rightarrow\R$ be a
superlinear map. Assume that there is a real number $c>0$ such that for every $s_1,s_2\in\mcal{S}$, either $f(s_1+s_2)=f(s_1)+f(s_2)$
or $f(s_1+s_2)\geq f(s_1)+f(s_2)+c$. Let $\mcal{C}$ be a rational polyhedral cone in $\Int\mcal{S}_\R$.
Then $f_{|\mcal{C}}$ is rationally piecewise linear.
\end{te}

\begin{cor}\label{PL}
Let $\mcal{S}\subset\N^r$ be a finitely generated monoid and let $f\colon\mcal{S}_\R\rightarrow\R$ be a
superlinear map such that $f(\mcal{S})\subset\Z$. Let $\mcal{C}$ be a rational polyhedral cone in $\Int\mcal{S}_\R$.
Then $f_{|\mcal{C}}$ is rationally piecewise linear.
\end{cor}

\begin{re}
In Theorem \ref{PLinitial} and Corollary \ref{PL}, instead of $\mcal{S}\subset\N^r$ we can assume that $\mcal{S}\subset\Q^r$ and that
$\mcal{S}_\R$ is strongly convex.
\end{re}

\begin{ex}
The condition $f(\mcal{S})\subset\Z$ in Corollary \ref{PL} is crucial. Let $\mcal{S}=\N(-1,1)+\N(1,0)\subset\R^2$ and let $x_1=(1,0)$ and $x_n=(1/2^n,1)$
for $n\geq2$. Let $\mcal{C}_n=\R_+x_n+\R_+x_{n+1}$ and observe that $\R_+^2=\bigcup_{n\geq1}\mcal{C}_n\cup\big(\R_+(0,1)+\R_+(-1,1)\big)$.
Define the sequences of positive rational numbers $f(x_n)$ and $\varepsilon_n$ as follows: set $f(x_2)=3$ and $\varepsilon_2=23/8$. Assume $f(x_{n-1})$
and $\varepsilon_{n-1}$ are defined; then set $f(x_n)=(f(x_{n-1})+\varepsilon_{n-1})/2$ and choose $\varepsilon_n$ such that
$f(x_n)-1/2^n<\varepsilon_n<\varepsilon_{n-1}$; we can always arrange $\lim_{n\rightarrow\infty}\varepsilon_n\in\Q$.
Set $f(-1,1)=\lim_{n\rightarrow\infty}f(x_n)-1$, $f(0,1)=\lim_{n\rightarrow\infty}f(x_n)$, $f(x_1)=1/2$,
$f(\alpha x_n+\beta x_{n+1})=\alpha f(x_n)+\beta f(x_{n+1})$ and $f(-\alpha,\alpha+\beta)=\alpha f(-1,1)+\beta f(0,1)$ for $\alpha,\beta\geq0$.
Obviously $f(\mcal{S})\subset\Q$ and it is easy to check that $f$ is superlinear and continuous, but it is not PL on the subcone
$\bigcup_{n\geq2}\mcal{C}_n\cup\R_+(0,1)$.
\end{ex}

Throughout this section I will use without explicit mention basic properties of closed cones, see \cite[Section 6.3]{Deb01}.

\begin{lem}\label{local3}
Let $\mcal{S}=\N^{r+1}$ and let $f\colon\mcal{S}_\R\rightarrow\R$ be a
superlinear map. Assume that there is a real number $c>0$ such that for every $s_1,s_2\in\mcal{S}$, either $f(s_1+s_2)=f(s_1)+f(s_2)$
or $f(s_1+s_2)\geq f(s_1)+f(s_2)+c$. Let $x=(1,x_1,\dots,x_r)\in\Int\mcal{S}_\R$ and
let $R$ be a ray in $\mcal S_\R$ not containing $x$.

Then there exists a ray $R'\subset\R_+x+R$ not containing $x$ such that the map $f|_{\R_+x+R'}$ is linear.
\end{lem}
\begin{proof}
By induction, I assume Theorem \ref{PLinitial} when $\dim\mcal{S}_\R=r$.

The proof consists of three parts. In Steps 2-8 I assume the components of $x$ are linearly independent over $\Q$. In Step 9 I assume that $x$ is
a rational point while the remaining case when $x$ is a non-rational point which belongs to a rational hyperplane is settled in Step 10.\\[2mm]
\noindent{\em Step 1}:
Let $H$ be any $2$-plane not contained in a rational hyperplane.
Points of the form $(1,z_1,\dots,z_r)$, where $1,z_1,\dots,z_r$ are linearly independent over $\Q$, are dense on the line $L=H\cap(z_0=1)$.
Otherwise there would exist an open neighbourhood $U$ on $L$ such that for each point $z\in U$ there is a rational hyperplane $H_z\supset\R z$.
But the set of rational hyperplanes is countable.

On the other hand, fix a rational point $t\in\R^{r+1}\backslash H$ and observe rational hyperplanes containing $\R_+t$.
I claim that the set of points which are intersections of those hyperplanes and the line $L$ are dense on $L$.
To see this, let $y=(1,y_1,\dots,y_r)$ be a point in $H$ and let $\mcal{A}=(\alpha_0z_0+\dots+\alpha_rz_r=0)$ be any hyperplane containing $y$
and $t$. $H$ is given as a solution of a system of $r-1$ linear equations in $z_0,\dots,z_r$,
thus $y$ is a solution of a system of $r$ linear equations and the components of $y$ are linear functions in $\alpha_0,\dots,\alpha_r$,
where $\alpha_i$ are linearly dependent over $\Q$ (since $t\in\mcal{A}$). Therefore, without loss of generality,  wiggling $\alpha_i$ for $i<r$
we can obtain a point $y'\in L$ arbitrarily close to $y$ which belongs to a rational hyperplane
$\mcal{A}'=(\alpha_0'z_0+\dots+\alpha_r'z_r=0)$ containing $t$.
Furthermore, if $H$ contains a rational point $t_0$, then $y'$ cannot belong to a rational plane $\widetilde{\mcal{A}}$ of dimension $<n-1$
since otherwise $H$ would be contained in a rational hyperplane generated by $\widetilde{\mcal{A}}$ and $t_0$.
\\[2mm]
\noindent{\em Step 2}:
In Steps 2-8 I assume that the real numbers $1,x_1,\dots,x_r$ are linearly independent over $\Q$.

For a real number $\alpha$ let $\|\alpha\|:=\min\{\alpha-\lfloor\alpha\rfloor,\lceil\alpha\rceil-\alpha\}$. By \cite[Chapter I, Theorem VII]{Cas57}
there are infinitely many positive integers $q$ such that
\begin{equation}\label{cassels11}
\|qx_i\|<q^{-1/r}
\end{equation}
for all $i$. Fix such a $q$ big enough so that the ball of radius $1/q$ centred at $x$ is contained in $\Int\mcal{S}_\R$ and so that
$q^{1/r}>r$, and in particular $\sum\|qx_i\|<1$. Let $p_i$ be positive integers with $|qx_i-p_i|<q^{-1/r}$.

Let
$$\widehat p_i=\begin{cases}
\lfloor qx_i\rfloor & \textrm{if }p_i=\lceil qx_i\rceil\\
\lceil qx_i\rceil & \textrm{if }p_i=\lfloor qx_i\rfloor.
\end{cases}$$
Let $e_0,e_1,\dots,e_r$ be the standard basis of $\R^{r+1}$. Set
$$u_0=qe_0+\sum p_ie_i$$
and
$$u_i=qe_0+\sum_{j\neq i}p_je_j+\widehat p_ie_i$$
for $i=1,\dots,r$.
From \eqref{cassels11} we have
\begin{equation}\label{cassels211}
\|x-u_0/q\|_{\infty}<q^{-1-1/r}.
\end{equation}
It is easy to see that $u_0,\dots,u_r$ are linearly independent and that
\begin{equation}\label{relation}
\Big(1-\sum\|qx_i\|\Big)u_0+\sum\|qx_i\|u_i=qx.
\end{equation}

Assume that for every open cone $U$ containing $x$ the map $f_{|U}$ {\em is not linear\/}. Then in Steps 3-7 I will prove that for all $q\gg0$ satisfying
\eqref{cassels11} we have
\begin{equation}\label{keyequation}
f(x)=\Big(1-\sum\|qx_i\|\Big)f(u_0/q)+\sum\|qx_i\|f(u_i/q)+e_q,
\end{equation}
where $e_q\geq c(1-\sum\|qx_i\|)/q$. I will then derive a contradiction in Step 8.
\\[2mm]
\noindent{\em Step 3}:
Let $\mcal{K}=\sum_{i\geq0}\R_+u_i$ and $\mcal{K}_i=\R_+x+\sum_{j\neq i}\R_+u_j$ for $i=0,\dots,r$; observe that $\mcal{K}=\bigcup_{i\geq0}\mcal{K}_i$.
Define the sequences $v_n\in\N^{r+1}$ and $j_n\in\N$ as follows: set $v_0=\sum_{i\geq0}u_i$.
If $v_n$ is defined then, since the components of $x$ are linearly independent over $\Q$, there is a unique $j_n\in\{0,\dots,r\}$ such that
$v_n$ belongs to the {\em interior\/} of $\mcal{K}_{j_n}$. Set $v_{n+1}=v_n+u_{j_n}$.
Define the sequence of non-negative real numbers $e_n$ by
$$f(v_{n+1})=f(v_n)+f(u_{j_n})+e_n.$$
\noindent{\em Step 4}:
In this step I assume that for all $n\geq n_0$ with $j_n=0$ we have $e_n\geq c$. Then we have
\begin{equation}\label{step}
f(v_n)=\sum_{i=0}^r\alpha_i^{(n)}f(u_i)+e^{(n)},
\end{equation}
where $\alpha_i^{(n)}\in\N$ and $e^{(n)}\geq c(\alpha_0^{(n)}-n_0)$. Observe that $v_n=\sum_{i=0}^r\alpha_i^{(n)}u_i$,
and therefore from Lemma \ref{converge} we have
$$qx=\lim_{n\rightarrow\infty}v_n/n=\sum_{i=0}^r\lim_{n\rightarrow\infty}(\alpha_i^{(n)}/n)u_i.$$
Since $u_i$ are linearly independent, from \eqref{relation} we obtain
\begin{align*}
\lim_{n\rightarrow\infty}\alpha_0^{(n)}/n&=1-\sum\|qx_i\|
\intertext{and}
\lim_{n\rightarrow\infty}\alpha_i^{(n)}/n&=\|qx_i\|\quad\text{for}\quad i>0.
\end{align*}
Dividing \eqref{step} by $n$, taking a limit when $n\rightarrow\infty$ and using continuity of $f$ and Lemma \ref{converge} we obtain
$$f(qx)=\Big(1-\sum\|qx_i\|\Big)f(u_0)+\sum\|qx_i\|f(u_i)+\widehat e_q,$$
where $\widehat e_q\geq c(1-\sum\|qx_i\|)$. Dividing now by $q$ we get \eqref{keyequation}.
\\[2mm]
\noindent{\em Step 5}: In Steps 5-7 I assume there are infinitely many $n$ with $j_n=0$ and $e_n=0$.
Then by Lemma \ref{linear2} the map $f|_{\R_+v_n+\R_+u_0}$ is linear for each such $n$ (observe that when $r=1$
this finishes the proof since then $x\in\Int(\R_+v_n+\R_+u_0)$). But then we have
$$f(v_n/n+u_0)=f(v_n/n)+f(u_0),$$
so letting $n\rightarrow\infty$ and using Lemma \ref{converge} we get
$$f(qx+u_0)=f(qx)+f(u_0),$$
thus the map $f|_{\R_+x+\R_+u_0}$ is linear by Lemma \ref{linear2}.

Let us first prove that there is an $(r+1)$-dimensional polyhedral cone $\mcal{C}_{r+1}$ such that $\R_+x+\R_+u_0\subset\mcal{C}_{r+1}$,
$(\R_+x+\R_+u_0)\cap\Int\mcal{C}_{r+1}\neq\emptyset$ and
$f_{|\mcal{C}_{r+1}}$ is linear.
Let $t\in\Int\mcal{S}_\R\backslash(\R x+\R u_0)$ be a rational point.
By Step 1 there is a rational hyperplane $\mcal{H}\ni t$ such that
there is a nonzero $w\in\mcal{H}\cap\relint(\R_+x+\R_+u_0)$, and there does not exist a rational plane of dimension $<n-1$ containing $w$.
By Theorem \ref{PLinitial} applied to $\mcal{H}\cap\mcal{S}_\R$ there is an $r$-dimensional cone
$\mcal{C}_r=\sum_{i=1}^r\R_+h_i\subset\mcal{H}\cap\mcal S_\R$
such that $w\in\relint\mcal{C}_r$ and $f_{|\mcal{C}_r}$ is linear.
Set $\mcal{C}_{r+1}=\mcal{C}_r+\R_+x+\R_+u_0$.
Now if $w=\sum\mu_ih_i$ with all $\mu_i>0$, since $f$ is linear on $\mcal{C}_r$ we have
\begin{multline*}
f\Big(x+u_0+\sum\mu_ih_i\Big)=f(x+u_0+w)\\
=f(x)+f(u_0)+f(w)=f(x)+f(u_0)+\sum\mu_if(h_i),
\end{multline*}
so the map $f_{|\mcal{C}_{r+1}}$ is linear by Lemma \ref{linear2}.\\[2mm]
\noindent{\em Step 6}:
Let $\mcal{C}=\R_+g_1+\dots+\R_+g_m$ be any $(r+1)$-dimensional polyhedral cone containing $x$ such that $f_{|\mcal{C}}$ is linear
and let $\ell$ be the linear extension of $f_{|\mcal{C}}$ to $\R^{r+1}$.
Assume that for a point $h\in\mcal{S}_\R$ we have $f|_{\R_+h}=\ell|_{\R_+h}$. There are real numbers $\lambda_i$ such that
$$h=\sum\nolimits_i\lambda_ig_i.$$
Then setting $e:=\sum(1+|\lambda_i|)g_i+h=\sum(1+|\lambda_i|+\lambda_i)g_i\in\mcal{C}$ we have
\begin{align*}
f(e)&=\ell\Big(\sum(1+|\lambda_i|+\lambda_i)g_i\Big)=\sum(1+|\lambda_i|+\lambda_i)\ell(g_i)\\
&=\sum(1+|\lambda_i|)\ell(g_i)+\ell(h)=\sum(1+|\lambda_i|)f(g_i)+f(h),
\end{align*}
so $f$ is linear on the cone $\mcal{C}+\R_+h$ by Lemma \ref{linear2}. Therefore the set $\widehat{\mcal{C}}=\{z\in\mcal{S}_\R:f(z)=\ell(z)\}$
is an $(r+1)$-dimensional closed cone.
\\[2mm]
\noindent{\em Step 7}:
Since $f$ is not linear in any open neighbourhood of $x$ we have $x\notin\Int\widehat{\mcal{C}}$.
Therefore there is a tangent hyperplane $T$ to $\widehat{\mcal{C}}$ containing $x$. Let $W_1$ and $W_2$ be the half-spaces such that $W_1\cap W_2=T$,
$W_1\cup W_2=\R^{r+1}$ and $\widehat{\mcal{C}}\subset W_1$. Since $(\R_+x+\R_+u_0)\cap\Int\widehat{\mcal{C}}\neq\emptyset$ we must have
$(\R x+\R u_0)\cap W_2\neq\emptyset$.

By Step 1 applied to the $2$-plane $\R x+\R u_0$, for every non-negative $\varepsilon<q^{-1-1/r}-\max\{\|qx_i\|/q\}$ let
$$x_{\varepsilon}=(1,x_{\varepsilon,1},\dots,x_{\varepsilon,r})\in(\R x+\R u_0)\cap W_2$$
be such that $0<\|x-x_{\varepsilon}\|_{\infty}\leq\varepsilon$ and the components of $x_{\varepsilon}$ are linearly independent over $\Q$.
The map $f|_{\R_+u_0+\R_+x_{\varepsilon}}$ is not linear since otherwise we would have $f(x_{\varepsilon})=\ell(x_{\varepsilon})$.
Observe that $|qx_{\varepsilon,i}-p_i|<q^{-1/r}$ for every $i$.
Then as in Step 4 we have
$$f(qx_{\varepsilon})=\Big(1-\sum\|qx_{\varepsilon, i}\|\Big)f(u_0)+\sum\|qx_{\varepsilon,i}\|f(u_i)+\widehat e_q,$$
where $\widehat e_q\geq c(1-\sum\|qx_{\varepsilon,i}\|)$. Finally dividing by $q$ and letting $\varepsilon\rightarrow0$ we obtain \eqref{keyequation}.
\\[2mm]
\noindent{\em Step 8}:
Therefore for all $q\gg0$ satisfying \eqref{cassels11} we have \eqref{keyequation}.
Then since $f$ is locally Lipschitz around $x$ there is a constant $L>0$ such that
\begin{align*}
c(q^{1/r}&-r)q^{-1-1/r}<c\Big(1-\sum\|qx_i\|\Big)/q\leq e_q\\
&=\bigg(f(x)-\Big(1-\sum\|qx_i\|\Big)f(u_0/q)-\sum\|qx_i\|f(u_i/q)\bigg)\\
&=\Big(\big(f(x)-f(u_0/q)\big)+\sum\|qx_i\|\big(f(u_0/q)-f(u_i/q)\big)\Big)\\
&\leq L\|x-u_0/q\|_{\infty}+\sum\|qx_i\|L\|u_0/q-u_i/q\|_{\infty}\\
&<Lq^{-1-1/r}+\sum_{i=1}^rq^{-1/r}Lq^{-1}=L(r+1)q^{-1-1/r},
\end{align*}
where I used \eqref{cassels11} and \eqref{cassels211}. Hence $L>c(q^{1/r}-r)/(r+1)$ for $q\gg0$, a contradiction.

Thus if $1,x_1,\dots,x_r$ are linearly independent over $\Q$ then there is an open cone containing $x$ where $f$ is linear, so the lemma follows.
In particular there are linearly independent rational rays $R_1,\dots,R_{r+1}\subset\Int\mcal{S}_\R$ such that
$R\subset\Int(R_1+\dots+R_{r+1})$ and  the map $f|_{R_1+\dots+R_{r+1}}$ is linear.
\\[2mm]
\noindent{\em Step 9}:
Assume now that $x$ is a rational point. By induction I assume there does not exist a rational hyperplane containing $\R_+x$ and $R$.
By clearing denominators I can assume $x=(\kappa,x_1,\dots,x_r)$ where $\kappa,x_i\in\N$.

Fix $q$ big enough so that the ball of radius $1/q$ centred at $x$ is contained in $\Int\mcal{S}_\R$ and so that $q^{1/r}>r$.
Fix a positive $\varepsilon<q^{-1-1/r}$. By Step 1 there is a point
$$x_{\varepsilon}=(\kappa,x_{\varepsilon,1},\dots,x_{\varepsilon,r})\in\R_+x+R$$
such that $\|x-x_{\varepsilon}\|_{\infty}\leq\varepsilon$
and the components of $x_{\varepsilon}$ are linearly independent over $\Q$. Set $u_0=qx$, define integers $p_i$ and $\widehat{p}_i$ with respect to
$x_{\varepsilon}$ as in Step 2
and set
$$u_i=q\kappa e_0+\sum_{j\neq i}p_je_j+\widehat p_ie_i$$
for $i=1,\dots,r$.
Then $u_0,\dots,u_r$ are linearly independent and we have
$$\Big(1-\sum\|qx_{\varepsilon,i}\|\Big)u_0+\sum\|qx_{\varepsilon,i}\|u_i=qx_{\varepsilon}.$$

With respect to $x_{\varepsilon}$ define the sequences $v_n\in\N^{r+1}$ and $(j_n,e_n)\in\N\times\R_+$ as in Step 3.
Assume that for all $n\geq n_0$ with $j_n=0$ we have
$e_n\geq c$. Then as in Step 4 we obtain
\begin{equation}\label{wrinkle}
f(qx_{\varepsilon})=\Big(1-\sum\|qx_{\varepsilon,i}\|\Big)f(u_0)+\sum\|qx_{\varepsilon,i}\|f(u_i)+\widehat e_q,
\end{equation}
where $\widehat e_q\geq c(1-\sum\|qx_{\varepsilon,i}\|)$. If \eqref{wrinkle} stands for every $\varepsilon<q^{-1-1/r}$ then dividing \eqref{wrinkle}
by $q$ and letting $\varepsilon\rightarrow0$ we get
$$f(x)=f(x)+e_q$$
where $e_q\geq c/q$, a contradiction.
Therefore there is a positive $\varepsilon<q^{-1-1/r}$ such that there are infinitely many $n$ with $j_n=0$ and $e_n=0$.
But then as in Step 5 we have that the map $f|_{\R_+x+\R_+x_{\varepsilon}}$ is linear and we are done.\\[2mm]
\noindent{\em Step 10}:
Assume finally that $x$ is a non-rational point contained in a rational hyperplane; let $H$ be a rational plane of the smallest dimension
containing $x$ and set $k=\dim H$. Let $R=\R_+v$.

By Theorem \ref{PLinitial} there is a rational cone $\mcal{C}=\sum_{i=1}^k\R_+g_i\subset H$ with $g_i$ being rational points
such that $f_{|\mcal{C}}$ is linear and $x\in\relint\mcal{C}$, or equivalently $x=\sum\lambda_ig_i$
with all $\lambda_i>0$. Take a rational point $y=\sum_{i=1}^kg_i$.
Then by Step 9 there is a point $x'=\alpha y+\beta v$ with $\alpha,\beta>0$ such that
the map $f|_{\R_+y+\R_+x'}$ is linear. Now we have
$$f\Big(\sum g_i+x'\Big)=f(y+x')=f(y)+f(x')=\sum f(g_i)+f(x'),$$
so the map $f|_{\mcal{C}+\R_+x'}$ is linear by Lemma \ref{linear2}. Taking $\mu=\max\limits_i\{\alpha/(\lambda_i\beta)\}$
and setting $\hat v=\mu x+v\in\relint(\R_+x+R)$, it is easy to check that
$$\hat v=\sum(\mu\lambda_i-\alpha/\beta)g_i+x'/\beta\in\mcal{C}+\R_+x',$$
so the map $f|_{\R_+x+\R_+\hat v}$ is linear.
\end{proof}

\begin{lem}\label{converge}
Assume the notation from Lemma \ref{local3}. Then
$$\lim_{n\rightarrow\infty}v_n/n=qx.$$
\end{lem}
\begin{proof}
I work with the standard scalar product $\langle\cdot,\cdot\rangle$ and the induced Euclidean norm $\|\cdot\|$; denote $w_n=v_n/(n+r+1)$.
It is enough to prove $\lim_{n\rightarrow\infty}w_n=qx$.
By restricting to the hyperplane $(z_0=q)$ in $\R^{r+1}$ I assume the ambient space is $\R^r$.
\begin{figure}[htb]
\begin{center}
\includegraphics[width=0.62\textwidth]{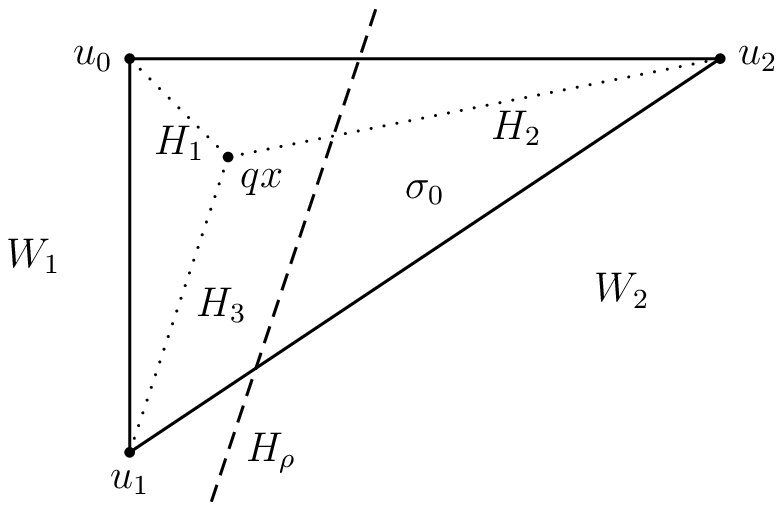}
\end{center}
\end{figure}
\\[2mm]
\noindent{\em Step 1}:
Let $\sigma$ denote the simplex with vertices $u_0,\dots,u_r$ and let $d=\sqrt2$ be the diameter of $\sigma$. For each $i$, let $\sigma_i$ be the
simplex with vertices $qx$ and $u_j$ for $j\neq i$.
The points $w_n$ belong to $\sigma$ and
$$w_{n+1}=\frac{1}{n+r+2}\big((n+r+1)w_n+u_{j_n}\big),$$
so we immediately get
\begin{equation}\label{bounds}
\|w_n-w_{n+1}\|\leq d/(n+r+2).
\end{equation}
For $\alpha=1,\dots,{r+1\choose2}$ let $H_{\alpha}$ be all hyperplanes containing the faces of the simplices $\sigma_i$ which contain $qx$.
\\[2mm]
\noindent{\em Step 2}:
Let us prove that for each $\alpha$ and for each $n$,
\begin{equation}\label{notintersect}
\dist\{w_{n+1},H_{\alpha}\}<\dist\{w_n,H_{\alpha}\}
\end{equation}
if the segment $[w_n,w_{n+1}]$ does not intersect $H_{\alpha}$, and otherwise
\begin{equation}\label{intersect}
\dist\{w_{n+1},H_{\alpha}\}<d/(n+r+2).
\end{equation}
To this end, if $H_{\alpha}$ contains $u_{j_n}$, then obviously $\dist\{w_{n+1},H_{\alpha}\}<\dist\{w_n,H_{\alpha}\}$. If $H_{\alpha}$ does not
contain $u_{j_n}$, then $u_{j_n}$ and $w_n$ are on different sides of $H_{\alpha}$. Now if the segment $[w_n,w_{n+1}]$ does not intersect $H_{\alpha}$
then \eqref{notintersect} is obvious, whereas otherwise \eqref{intersect} follows from \eqref{bounds}.
\\[2mm]
\noindent{\em Step 3}:
Now assume that for each $\alpha$, there are infinitely many segments $[w_n,w_{n+1}]$ intersecting $H_{\alpha}$. Then from \eqref{notintersect} and
\eqref{intersect} we get
$$\lim_{n\rightarrow\infty}\dist\{w_n,H_\alpha\}=0$$
and thus the sequence $w_n$ accumulates on each of the hyperplanes
$H_{\alpha}$. But $\bigcap_{\alpha}H_{\alpha}=\{qx\}$, so $\lim\limits_{n\rightarrow\infty}w_n=qx$.\\[2mm]
\noindent{\em Step 4}:
Finally let $\alpha_0$ be such that no segment $[w_n,w_{n+1}]$ intersects $H_{\alpha_0}$ for all $n\geq n_0$ and
$\lim\limits_{n\rightarrow\infty}\dist\{w_n,H_{\alpha_0}\}=\rho>0$ (the sequence $\dist\{w_n,H_{\alpha_0}\}$ converges by \eqref{notintersect}).
Therefore there is a hyperplane $H_\rho$ parallel to $H_{\alpha_0}$ such that $\dist\{H_\rho,H_{\alpha_0}\}=\rho$ and the sequence $w_n$
accumulates on $H_\rho$; let $W_1$ and $W_2$ be the two half-spaces such that $W_1\cup W_2=\R^r$ and $W_1\cap W_2=H_\rho$.
Relabelling we can assume $u_0,\dots,u_{r-1},qx\in W_1$ and $w_n,u_r\in W_2$ for all $n\geq n_0$; observe that then
$u_{j_n}\in\{u_0,\dots,u_{r-1}\}$ for all $n\geq n_0$.

By change of coordinates I may assume that $H_{\alpha_0}$ contains the origin. Fix a nonzero vector $a$ perpendicular to $H_{\alpha_0}$
such that $W_2\subset\{z\in\R^r:\langle a,z\rangle\geq0\}$. Since $W_2\cap H_{\alpha_0}=\emptyset$ the linear function $\langle a,\cdot\rangle$
attains its minimum $m>0$ on the compact set $W_2\cap\sigma$. Then since $\langle a,u_{j_n}\rangle\leq0$ for $n\geq n_0$ we have
\begin{align*}
\dist\{w_n,H_{\alpha_0}\}&-\dist\{w_{n+1},H_{\alpha_0}\}=\frac{\langle a,w_n-w_{n+1}\rangle}{\|a\|}\\
&=\frac{\langle a,w_n-u_{j_n}\rangle}{(n+r+2)\|a\|}\geq\frac{m}{(n+r+2)\|a\|},
\end{align*}
and therefore
$$\dist\{w_{n_0},H_{\alpha_0}\}\geq\frac{m}{\|a\|}\sum_{n\geq n_0}\frac{1}{n+r+2}=+\infty,$$
a contradiction.
\end{proof}

\begin{cor}\label{everyplane}
Let $\mcal{S}\subset\N^{r+1}$ be a finitely generated monoid and let $f\colon\mcal{S}_\R\rightarrow\R$ be a superlinear map.
Assume there is a real number $c>0$ such that for every $s_1,s_2\in\mcal{S}$, either $f(s_1+s_2)=f(s_1)+f(s_2)$
or $f(s_1+s_2)\geq f(s_1)+f(s_2)+c$. Let $\mcal{C}$ be a polyhedral cone in $\Int\mcal S_\R$.

Then for every $2$-plane $H$ the map $f_{|\mcal{C}\cap H}$ is piecewise linear.
\end{cor}
\begin{proof}
If $\mcal{C}=\bigcup\mcal{C}_i$ is a finite subdivision of $\mcal{C}$ into rational simplicial cones, then $f_{|\mcal{C}\cap H}$ is PL
if and only if $f_{|\mcal{C}_i\cap H}$ is PL for every $i$, so I assume $\mcal{C}$ is simplicial.
Take a basis $g_1,\dots,g_{r+1}\in\mcal{S}$ of $\mcal{C}$, set $s:=\sum g_i$ and let $0<\alpha\ll1$ be a rational number such that
$g_i+\alpha(g_i-s)\in\Int\mcal{S}_\R$ for all $i$. Take $g_i'\in\mcal{S}\cap\R_+\big(g_i+\alpha(g_i-s)\big)$. It is easy to check that
$g_i'$ are linearly independent and that $\mcal{C}\subset\Int(\sum\R_+g_i')$. Therefore I can assume $\mcal{S}=\N^{r+1}$.

By Lemma \ref{local3}, for every ray $R\subset\mcal{C}\cap H$ there is a polyhedral cone
$\mcal{C}_R$ with $R\subset\mcal{C}_R\subset\mcal{C}\cap H$ such that there is a polyhedral decomposition
$\mcal{C}_R=\mcal{C}_{R,1}\cup\mcal{C}_{R,2}$ with $f_{|\mcal{C}_{R,1}}$ and $f_{|\mcal{C}_{R,2}}$
being linear maps, and if $R\subset\relint(\mcal{C}\cap H)$, then $R\subset\relint\mcal{C}_R$.

Let $\|\cdot\|$ be the standard Euclidean norm and let $S=\{z\in\R^{r+1}:\|z\|=1\}$ be the unit sphere. Restricting to the compact set
$S\cap\mcal{C}\cap H$ we can choose finitely many polyhedral cones $\mcal{C}_i$ with
$\mcal{C}\cap H=\bigcup\mcal{C}_i$ such that each $f_{|\mcal{C}_i}$ is PL. But then $f_{|\mcal{C}\cap H}$ is PL.
\end{proof}

\begin{lem}\label{piecewise}
Let $f$ be a superlinear function on a polyhedral cone $\mcal{C}\subset\R^{r+1}$ with $\dim\mcal{C}=r+1$ such that for every
$2$-plane $H$ the function $f_{|H\cap\mcal{C}}$ is piecewise linear. Then $f$ is piecewise linear.
\end{lem}
\begin{proof}
I will prove the lemma by induction on the dimension.\\[2mm]
\noindent{\em Step 1}:
Fix a ray $R\subset\mcal{C}$. In this step I prove that for any ray $R'\subset\mcal{C}$ there is an $(r+1)$-dimensional cone
$\mcal{C}_{(r+1)}\subset\mcal{C}$ containing $R$ such that the map $f_{|\mcal{C}_{(r+1)}}$ is linear and $\mcal{C}_{(r+1)}\cap(R+R')\neq R$.

Let $H_r\supset(R+R')$ be any hyperplane. By induction there is an $r$-dimensional polyhedral cone
$\mcal{C}_{(r)}=\sum_{i=1}^r\R_+e_i\subset H_r\cap\mcal{C}$ containing $R$ such that $f_{|\mcal{C}_{(r)}}$
is linear and $\mcal{C}_{(r)}\cap(R+R')\neq R$. Set $e_0=e_1+\dots+e_r$.
Let $H_2$ be a $2$-plane such that $H_2\cap H_r=\R_+e_0$. Since $f_{|H_2\cap\mcal{C}}$ is
PL, there is a point $e_{r+1}\in H_2\cap\mcal{C}$ such that $f|_{\R_+e_0+\R_+e_{r+1}}$ is linear.
Set $\mcal{C}_{(r+1)}=\R_+e_1+\dots+\R_+e_{r+1}$. Then we have
$$f\Big(\sum e_i\Big)=f(e_0+e_{r+1})=f(e_0)+f(e_{r+1})=\sum f(e_i),$$
so the map $f_{|\mcal{C}_{(r+1)}}$ is linear by Lemma \ref{linear2}.
Observe that choosing $e_{r+1}$ appropriately we can ensure that the cone $\mcal{C}_{(r+1)}$ is contained in either
of the half-spaces into which $H_r$ divides $\R^{r+1}$.\\[2mm]
\noindent{\em Step 2}:
Fix a ray $R\subset\mcal{C}$ and let $\mcal{C}_{(r+1)}$ be any $(r+1)$-dimensional cone such that $f_{|\mcal{C}_{(r+1)}}$ is linear.
Let $\ell$ be the linear extension of $f_{|\mcal{C}_{r+1}}$ to $\R^{r+1}$. Let $\widehat{\mcal{C}}=\{z\in\mcal{C}:f(z)=\ell(z)\}$;
it is a closed cone by Step 6 of the proof of Lemma \ref{local3}.

I claim $\widehat{\mcal{C}}$ is a locally polyhedral cone (and thus polyhedral). Otherwise, fix a boundary ray $R_{\infty}$ and let $H$ be
any hyperplane containing $R_{\infty}$ such that
$H\cap\Int\widehat{\mcal{C}}\neq\emptyset$. Let $R_n$ be a sequence of boundary rays which converge to $R_{\infty}$ and they are all on the
same side of $H$.

Let $T\supset R_{\infty}$ be any hyperplane tangent to $\widehat{\mcal{C}}$. Fix an $(r-1)$-plane $H_{r-1}\subset T$ containing $R_{\infty}$
and let $H_{r-1}^\perp$ be the unique $2$-plane orthogonal to $H_{r-1}$.
For each $n$ consider a hyperplane $H_r^{(n)}$ generated by $H_{r-1}$ and $R_n$ (if $R_n\subset H_{r-1}$ then we can finish by induction on the
dimension). Let $\|\cdot\|$ be the standard Euclidean norm and let $S=\{z\in\R^{r+1}:\|z\|=1\}$ be the unit sphere. The set of points
$\bigcup_{n\in\N}\big(S\cap H_{r-1}^\perp\cap H_r^{(n)}\big)$
has a limit $P_{\infty}$ on the circle $S\cap H_{r-1}^\perp$ and let $H_r^{(\infty)}$ be the hyperplane generated by $H_{r-1}$
and $P_{\infty}$; without loss of generality I can assume all $R_n$ are on the same side of $H_r^{(\infty)}$.

Now by the construction in Step 1, there is an $(r+1)$-dimensional cone $\mcal{C}_\infty$ such that $\mcal{C}_\infty\cap H_r^{(\infty)}$ is a face of
$\mcal{C}_\infty$, $f_{|\mcal{C}_\infty}$ is linear and $\mcal{C}_\infty$ intersects hyperplanes $H_r^{(n)}$ for all $n\gg0$.
In particular $R_n\subset\mcal{C}_\infty$ for all $n\gg0$ and $\Int\mcal{C}_\infty\cap\widehat{\mcal{C}}\neq\emptyset$.
Let $w\in\Int\mcal{C}_\infty\cap\widehat{\mcal{C}}$ and let $B\subset\Int\mcal{C}_\infty$ be a small
ball centred at $w$. Then the cone $B\cap\widehat{\mcal{C}}$ is $(r+1)$-dimensional (otherwise the cone $\widehat{\mcal{C}}$ would be
contained in a hyperplane) and thus $\mcal{C}_\infty\cap\widehat{\mcal{C}}$ is an $(r+1)$-dimensional cone.
Therefore the linear extension of $f_{|\mcal{C}_\infty}$ coincides with $\ell$ and thus $\mcal{C}_\infty\subset\widehat{\mcal{C}}$.
Since $R_n\not\subset\Int\widehat{\mcal{C}}$ we must have $R_n\subset\mcal{C}_\infty\cap H_r^{(\infty)}$,
and we finish by induction on the dimension.
\\[2mm]
\noindent{\em Step 3}:
Again fix a ray $R\subset\mcal{C}$. By Steps 1 and 2 there is a collection of $(r+1)$-dimensional polyhedral cones
$\{\mcal{C}_{\alpha}\}_{\alpha\in I_R}$ such that $R\subset\mcal{C}_{\alpha}\subset\mcal{C}$ for every $\alpha\in I_R$,
for every ray $R'\subset\mcal{C}$ there is $\alpha\in I_R$ such that $\mcal{C}_{\alpha}\cap(R+R')\neq R$
and for every two distinct $\alpha,\beta\in I_R$ the linear extensions of $f_{|\mcal{C}_{\alpha}}$ and $f_{|\mcal{C}_{\beta}}$ to $\R^{r+1}$
are not the same function. I will prove that $I_R$ is a finite set.

For each $\alpha\in I_R$ let $x_\alpha$ be a point in $\Int\mcal{C}_\alpha$ and let $H_{\alpha}=(R+\R_+x_\alpha)\cup(-R+\R_+x_\alpha)$.
Let $R_{\alpha}\subset H_\alpha$ be the unique ray orthogonal to $R$. Let $R^\perp$ be the hyperplane orthogonal to $R$.
For each $\alpha$ let $S\cap R^\perp\cap H_\alpha=\{Q_\alpha\}$. If there are infinitely many cones $\mcal{C}_{\alpha}$, then the set
$\{Q_\alpha:\alpha\in I_R\}$ has an accumulation point $Q_\infty$. Let $H_\infty=(R+\R_+Q_\infty)\cup(-R+\R_+Q_\infty)$, let
$H_n$ be a sequence in the set $\{H_{\alpha}\}$ such that $\lim\limits_{n\rightarrow\infty}Q_n=Q_\infty$ where $S\cap R^\perp\cap H_n=\{Q_n\}$,
and let $\mcal{C}_n$ be the corresponding cones in $\{\mcal{C}_{\alpha}\}$.

By assumptions of the lemma there is a point $y\in H_\infty$ such that $f|_{R+\R_+y}$ is linear. Let $x$ be a point on $R$ and let $\mcal{H}$ be any
hyperplane such that $\mcal{H}\cap(\R x+\R y)=\R(x+y)$. By induction there are $r$-dimensional polyhedral cones $\mcal{C}_1,\dots,\mcal{C}_k$
in $\mcal{H}\cap\mcal{C}$ such that $x+y\in\mcal{C}_i$ for all $i$, there is a small $r$-dimensional ball $B_{(r)}\subset\mcal{H}$ centred at $x+y$
such that $B_{(r)}\cap\mcal{C}=B_{(r)}\cap(\mcal{C}_1\cup\dots\cup\mcal{C}_k)$ and
the map $f_{|\mcal{C}_i}$ is linear for every $i$. Fix $i$ and let $g_{ij}$ be generators of $\mcal{C}_i$. Then
\begin{align*}
f\Big(\sum\nolimits_jg_{ij}+x+y\Big)&=\sum\nolimits_jf(g_{ij})+f(x+y)\\
&=\sum\nolimits_jf(g_{ij})+f(x)+f(y),
\end{align*}
so $f$ is linear on the cone $\widetilde{\mcal{C}}_i=\mcal{C}_i+\R_+x+\R_+y$ by Lemma \ref{linear2}. Therefore if we denote
$\widetilde{\mcal{C}}=\mcal{C}_1+\dots+\mcal{C}_k+\R_+x+\R_+y$, then $f_{|\widetilde{\mcal{C}}}$ is PL and
there is a small ball $B_{(r+1)}$ centred at $x+y$ such that $B_{(r+1)}\cap\mcal{C}=B_{(r+1)}\cap\widetilde{\mcal{C}}$.

Take a ball $B_{\varepsilon}$ of radius $\varepsilon\ll1$ centred at $x+y$ such that $x\notin B_\varepsilon$ and
$B_\varepsilon\cap\mcal{C}=B_\varepsilon\cap\widetilde{\mcal{C}}$. Since $\|Q_n-Q_\infty\|<\varepsilon$ for $n\gg0$,
then considering the subspace generated by $R,Q_n$ and $Q_\infty$ we obtain that $H_n$ intersects $\Int B_\varepsilon$ for $n\gg0$. Since
$\widetilde{\mcal{C}}=\bigcup\widetilde{\mcal{C}}_i$, there is an index $i_0$ such that $\widetilde{\mcal{C}}_{i_0}\cap\Int B_\varepsilon$
intersects infinitely many $H_n$. In particular, $\widetilde{\mcal{C}}_{i_0}\cap\Int\mcal{C}_n\neq\emptyset$ for
infinitely many $n$ and therefore $\widetilde{\mcal{C}}_{i_0}\cap\mcal{C}_n$ is an $(r+1)$-dimensional cone as in Step 2.
Thus for every such $n$ the linear extensions of
$f_{|\widetilde{\mcal{C}}_{i_0}}$ and $f_{|\mcal{C}_n}$ to $\R^{r+1}$ are the same since they coincide with the linear extension
of $f_{|\widetilde{\mcal{C}}_{i_0}\cap\mcal{C}_n}$, which is a contradiction and $I_R$ is finite.
\\[2mm]
\noindent{\em Step 4}:
Finally, we have that for every ray $R\subset\mcal{C}$ the map $f|_{\bigcup_{\alpha\in I_R}\mcal{C}_\alpha}$ is PL and
there is small ball $B_R$ centred at $R\cap S$ such that $B_R\cap\mcal{C}=B_R\cap\bigcup_{\alpha\in I_R}\mcal{C}_\alpha$.
There are finitely many open sets $\Int B_R$ which cover the compact set $S\cap\mcal{C}$ and therefore we can choose finitely many polyhedral cones
$\mcal{C}_i$ with $\mcal{C}=\bigcup\mcal{C}_i$ such that $f_{|\mcal{C}_i}$ is PL for every $i$. Thus $f$ is PL.
\end{proof}

\begin{proof}[Proof of Theorem \ref{PLinitial}]
By Corollary \ref{everyplane} and Lemma \ref{piecewise} the map $f_{|\mcal{C}}$ is PL; in other words
we can choose finitely many polyhedral cones $\mcal{C}_i$ with
$\mcal{C}=\bigcup\mcal{C}_i$ such that $f_{|\mcal{C}_i}$ is linear for each $i$.
We can assume the linear extensions of the maps $f_{|\mcal{C}_i}$ and $f_{|\mcal{C}_j}$ to $\R^r$ are not the same
by Step 6 of the proof of Lemma \ref{local3}.

Let $H$ be a hyperplane which contains a common $(r-1)$-dimensional face of cones $\mcal{C}_i$ and $\mcal{C}_j$ and assume $H$ is not rational.
Then similarly as in Step 1 of the proof of Lemma \ref{local3} there is a point $x\in\mcal{C}_i\cap\mcal{C}_j$ whose components are linearly independent
over $\Q$. By the proof of Lemma \ref{local3} there is an $r$-dimensional cone $\widetilde{\mcal{C}}$ such that $x\in\Int\widetilde{\mcal{C}}$ and
the map $f_{|\widetilde{\mcal{C}}}$ is linear. But then as in Step 2 of the proof of Lemma \ref{piecewise} the cones $\widetilde{\mcal{C}}\cap\mcal{C}_i$
and $\widetilde{\mcal{C}}\cap\mcal{C}_i$ are $r$-dimensional and linear extensions of $f_{|\mcal{C}_i}$ and $f_{|\mcal{C}_j}$
coincide since they are equal to the linear extension of $f_{|\widetilde{\mcal{C}}}$, a contradiction.
Therefore all $(r-1)$-dimensional faces of the cones $\mcal{C}_i$ belong to rational hyperplanes and thus $\mcal{C}_i$ are rational
cones. Thus the map $f_{|\mcal{C}}$ is $\Q$-PL.
\end{proof}

\section{b-Divisorial algebras}\label{section}

\begin{de}\label{def1}
Let $X$ be a variety and let $\mcal{S}$ be a submonoid of $\N^r$. If $\m\colon\mcal{S}\rightarrow\bDiv(X)$ is
a superadditive (respectively additive) map, the system of b-divisors $\m(\mcal{S})=\{\m(s)\}_{s\in\mcal{S}}$ is called {\em superadditive}
(respectively {\em additive\/}).

The system $\m(\mcal{S})$ (respectively the map $\m$) is called {\em bounded\/} if the following two conditions
are satisfied:
\begin{itemize}
\item there is a reduced divisor $F$ on $X$ such that $\Supp \m(s)_X\subset F$ for every $s\in\mcal{S}$, that is
$\m$ has {\em bounded support on $X$\/},

\item for every $s\in\mcal{S}$, the limit $\lim\limits_{\kappa\rightarrow\infty}\m(\kappa s)/\kappa$ exists in $\bDiv(X)_{\R}$.
\end{itemize}

If $\mcal{S}'$ is a truncation of a finitely generated monoid $\mcal{S}$, the superadditive system $\m(\mcal{S}')$ is called
a {\em truncation\/} of the system $\m(\mcal{S})$.
\end{de}

Let $\pi\colon X\rightarrow Z$ be a projective morphism of normal varieties and let $\m\colon\mcal{S}\rightarrow\bDiv(X)$ be
a bounded superadditive map such that $\OO_X(\m(s))$ is a coherent sheaf for all $s\in\mcal{S}$.
Let us consider a {\em b-divisorial $\mcal{S}$-graded $\OO_Z$-algebra}
$$R(X,\m(\mcal{S}))=\bigoplus_{s\in\mcal{S}}\pi_*\OO_X(\m(s)).$$
$R(X,\m(\mcal{S}))$ is canonically a graded subalgebra of $k(Z)[T_1,\dots,T_r]$.

Since here we are interested primarily in finite generation questions, I almost always assume that $Z$ is affine.

\begin{de}
Let $X$ be a variety, let $\mcal{S}$ be a monoid and let $\m\colon\mcal{S}\rightarrow\bMob(X)$ be a superadditive map.
Let $\F$ be a b-divisor on $X$ with $\lceil\F\rceil\geq0$.

We say the system $\m(\mcal{S})$ is {\em $\F$-saturated\/} (or that it satisfies the saturation condition with respect to $\F$)
if for all $s,s_1,\dots,s_n\in\mcal{S}$ such that $s=\xi_1s_1+\dots+\xi_ns_n$ for some non-negative {\em rational\/} numbers $\xi_i$, there is a model
$Y_{s,s_1,\dots,s_n}\rightarrow X$ such that for all models $Y\rightarrow Y_{s,s_1,\dots,s_n}$ we have
$$\Mob\lceil\xi_1\m(s_1)_Y+\dots+\xi_n\m(s_n)_Y+\F_Y\rceil\leq \m(s)_Y.$$
If the models $Y_{s,s_1,\dots,s_n}$ do not depend on $s,s_1,\dots,s_n$, we say the system $\m(\mcal{S})$ is {\em uniformly $\F$-saturated\/}.
\end{de}

\begin{re}
It is important to understand that the numbers $\xi_i$ in the previous definition are rational, and that $s$ is not merely
an integral combination of $s_i$. This fact is crucial in proofs.
\end{re}

\begin{lem}\label{satrays}
Let $X$ be a variety, let $\mcal{S}$ be a monoid and let $\m\colon\mcal{S}\rightarrow\bMob(X)$ be a superadditive map.
Let $\F$ be a b-divisor on $X$ with $\lceil\F\rceil\geq0$.
The system $\m(\mcal{S})$ is $\F$-saturated if and only if for all $s\in\mcal{S}$
and all positive integers $\lambda$ and $\mu$, there is a model $Y_{s,\lambda,\mu}\rightarrow X$ such that
for all models $Y\rightarrow Y_{s,\lambda,\mu}$ we have
$$\Mob\lceil(\lambda/\mu)\m(\mu s)_Y+\F_Y\rceil\leq \m(\lambda s)_Y.$$
\end{lem}
\begin{proof}
Necessity is clear. For sufficiency, fix $s,s_1,\dots,s_n\in\mcal{S}$ and fix non-negative rational numbers $\xi_i$ such that
$s=\xi_1s_1+\dots+\xi_ns_n$. Let $\lambda$ be a positive integer such that $\lambda\xi_i\in\N$ for all $i$.
Then on all models $Y$ higher than $Y_{s,1,\lambda}$ we have
\begin{multline*}
\Mob\lceil\xi_1\m(s_1)_Y+\dots+\xi_n\m(s_n)_Y+\F_Y\rceil\\
=\Mob\big\lceil(1/\lambda)\big(\lambda\xi_1\m(s_1)_Y+\dots+\lambda\xi_n\m(s_n)_Y\big)+\F_Y\big\rceil\\
\leq\Mob\lceil(1/\lambda)\m(\lambda s)_Y+\F_Y\rceil\leq \m(s)_Y.
\end{multline*}
Therefore we can take $Y_{s,s_1,\dots,s_n}:=Y_{s,1,\lambda}$.
\end{proof}

\begin{de}
Let $(X,\Delta)$ be a relative weak Fano klt pair projective over an affine variety $Z$ where $K_X+\Delta$ is $\Q$-Cartier, and
let $\mcal{S}\subset\N^r$ be a finitely generated monoid.

A {\em Shokurov algebra\/} on $X$ is the b-divisorial algebra $R(X,\m(\mcal{S}))$, where $\m\colon\mcal{S}\rightarrow\bMob(X)$ is a superadditive map
such that the system $\m(\mcal{S})$ is bounded and $\A(X,\Delta)$-saturated.
\end{de}

\begin{re}
If $\dim\mcal S_\R=1$, the previous definition reduces to the definition of the Shokurov algebra as given in \cite{Cor07}.
\end{re}

The next result says that saturation is preserved under restriction.

\begin{lem}\label{restriction}
Let $(X,\Delta)$ be a relative weak Fano pair projective over an affine variety $Z$ and let $S$ be a prime component in $\Delta$.
Let $\mcal{S}$ be a finitely generated monoid and assume the system of mobile b-divisors $\{\M_s\}_{s\in\mcal{S}}$
on $X$ is $(\A(X,\Delta)+S)$-saturated. Assume also that $S$ is not contained in
$\Supp\M_{sX}$ for any $s\in\mcal{S}$. Then the system
$\{\res_S\M_s\}_{s\in\mcal{S}}$ on $S$ is $\A(S,\Diff(\Delta-S))$-saturated.
\end{lem}
\begin{proof}
This is analogous to \cite[Lemma 2.3.43, Lemma 2.4.3]{Cor07}. In particular, the claim follows as soon as we have the surjectivity of
the restriction map
\begin{multline*}
H^0\big(Y,\big\lceil\sum\xi_i\M_{s_iY}+(\A(X,\Delta)+S)_Y\big\rceil\big)\\
\rightarrow H^0\big(S_Y,\big\lceil\sum\xi_i\M_{s_iY|S_Y}+\A(S,\Diff(\Delta-S))_{S_Y}\big\rceil\big)
\end{multline*}
for all $\xi_i\in\Q_+$ and all $s_i\in\mcal{S}$, on log resolutions $f\colon Y=Y_{s_1,\dots,s_n}\rightarrow X$
where $\M_{s_iY}$ is free for every $i$. The obstruction to surjectivity is the group
\begin{multline*}
H^1\big(Y,\big\lceil\sum\xi_i\M_{s_iY}+\A(X,\Delta)_Y\big\rceil\big)\\
=H^1\big(Y,K_Y+\big\lceil-f^*(K_X+\Delta)+\sum\xi_i\M_{s_iY}\big\rceil\big).
\end{multline*}
But this group vanishes by Kawamata-Viehweg vanishing since $-(K_X+\Delta)$ is nef and big and all $\M_{s_iY}$ are nef.
\end{proof}

\begin{de}
Let $\pi\colon X\rightarrow Z$ be a projective morphism of varieties, let $\mcal{S}\subset\N^r$ be a finitely generated monoid and let
$\delta\colon\mcal{S}\rightarrow\N$ be an additive map. Assume $\{\B_s\}_{s\in\mcal{S}}$ is a system of effective
$\Q$-b-divisors on $X$ such that
\begin{enumerate}
\item the system $\{\delta(s)\B_s\}_{s\in\mcal{S}}$ is superadditive and bounded,
\item for each $s\in\mcal{S}$ there is a divisor $\Delta_s$ on $X$ such that $K_X+\Delta_s$ is klt and
$\lim\limits_{\kappa\rightarrow\infty}(1/\kappa)\B_{\kappa sX}\leq\Delta_s$,
\item for each $s\in\mcal{S}$ there is a model $Y_s$ over $X$ and a mobile b-divisor $\M_s$ such that
$$\M_{sY}=\Mob\big(\delta(s)(K_Y+\B_{sY})\big)$$
for every model $Y$ over $Y_s$.
\end{enumerate}
Let $\m\colon\mcal{S}\rightarrow\bMob(X)$ be the superadditive map given by $\m(s)=\M_s$ for all $s\in\mcal{S}$.

If the system $\m(\mcal{S})$ is $\F$-saturated for some b-divisor $\F$ with $\lceil\F\rceil\geq0$,
we say the system $\m(\mcal{S})$ is {\em adjoint\/} and that the algebra $R(X,\m(\mcal{S}))$ is an {\em adjoint algebra\/} on $X$.
\end{de}

\section{Finite generation revisited}

\begin{lem}\label{truncation}
Let $X$ be a variety projective over an affine variety $Z$, let $\mcal{S}=\sum_{i=1}^n\N e_i$ be a monoid and let
$\m\colon\mcal{S}\rightarrow\bMob(X)$ be a superadditive map.
If there are positive integers $\kappa_1,\dots,\kappa_n$ and a truncation $\mcal{S}'=\sum_{i=1}^n\N\kappa_ie_i$ of $\mcal{S}$
such that $R(X,\m(\mcal{S}'))$ is finitely generated, then the algebra $R(X,\m(\mcal{S}))$ is finitely generated.
\end{lem}
\begin{proof}
It is enough to observe that $R(X,\m(\mcal{S}))$ is an integral extension of $R(X,\m(\mcal{S}'))$: for any $\varphi\in R(X,\m(\mcal{S}))$ we have
$\varphi^{\kappa_1\cdots\kappa_n}\in R(X,\m(\mcal{S}'))$. That concludes the proof.
\end{proof}

\begin{lem}\label{limit}
Let $X$ be a variety projective over an affine variety $Z$. Assume a monoid $\mcal{S}$ is finitely generated and let
$\m\colon\mcal{S}\rightarrow\bMob(X)$ be a superadditive map. If there exists a rational polyhedral refinement
$\mcal S_\R=\bigcup_{i=1}^k\Delta_i$ such that $\m_{|\Delta_i\cap\mcal{S}}$ is an additive map up to truncation for each $i$, then
the algebra $R(X,\m(\mcal{S}))$ is finitely generated.
\end{lem}
\begin{proof}
Let $\{e_{ij}:j\in I_i\}$ be a finite set of generators of $\Delta_i\cap\mcal{S}$ by Lemma \ref{gordan} and let $\kappa_{ij}$
be positive integers such that $\m|_{\sum_{j\in I_i}\N\kappa_{ij}e_{ij}}$ is additive for each $i$. Set $\kappa:=\prod_{i,j}\kappa_{ij}$ and let
$\mcal{S}'=\sum_{i,j}\N\kappa e_{ij}$ be a truncation of $\mcal{S}$.

Let $\tilde{e}=\sum_{i,j}\lambda_{ij}\kappa e_{ij}\in\Delta_i\cap\mcal{S'}$
for some $\lambda_{ij}\in\N$. Then $\sum_{i,j}\lambda_{ij}e_{ij}\in\Delta_i\cap\mcal{S}$ and thus there are
$\mu_j\in\N$ such that $\sum_{i,j}\lambda_{ij}e_{ij}=\sum_{j\in I_i}\mu_je_{ij}$. From here we have
$$\tilde{e}=\kappa\sum\nolimits_{j\in I_i}\mu_je_{ij}\in\sum\nolimits_{j\in I_i}\N\kappa e_{ij}$$
and therefore $\Delta_i\cap\mcal{S}'=\sum_{j\in I_i}\N\kappa e_{ij}$ is a truncation of $\sum_{j\in I_i}\N\kappa_{ij}e_{ij}$; in particular
$\m_{|\Delta_i\cap\mcal{S}'}$ is additive for each $i$.

I claim the algebra $R(X,\m(\mcal{S}'))$ is finitely generated, and thus the algebra $R(X,\m(\mcal{S}))$ is finitely generated by Lemma \ref{truncation}.
To this end let $Y\rightarrow X$ be a model such that $\m(\kappa e_{ij})$ for all $i,j$ descend to $Y$.
Let $s=\sum_{j\in I_i}\nu_{ij}\kappa e_{ij}\in\Delta_i\cap\mcal{S}'$ for some $i$ and some $\nu_{ij}\in\N$. Then
\begin{align*}
\m(s)&=\sum\nolimits_{j\in I_i}\nu_{ij}\m(\kappa e_{ij})=\sum\nolimits_{j\in I_i}\nu_{ij}\overline{\m(\kappa e_{ij})_Y}\\
&=\overline{\sum\nolimits_{j\in I_i}\nu_{ij}\m(\kappa e_{ij})_Y}=\overline{\m(s)_Y},
\end{align*}
and thus $\m(s)$ descends to $Y$ and
$$R(X,\m(\mcal{S}'))=\bigoplus_{s\in\mcal{S}'}H^0(Y,\m(s)_Y).$$
Fix $i$ and consider the free monoid $\widehat{\mcal{S}}_i=\bigoplus_{j\in I_i}\N\kappa e_{ij}$; the associated Cox ring
$R(Y,\{\m(\kappa e_{ij})_Y\}_{j\in I_i})$ is finitely generated by \cite[Lemma 2.8]{HK00}.
The canonical projection $\widehat{\mcal{S}}_i\rightarrow\Delta_i\cap\mcal{S}'$ gives the surjection
$$R(Y,\{\m(\kappa e_{ij})_Y\}_{j\in I_i})\rightarrow R(X,\m(\Delta_i\cap\mcal{S}')),$$
thus the algebra $R(X,\m(\Delta_i\cap\mcal{S}'))$ is finitely generated for each $i$. The set of generators of $R(X,\m(\Delta_i\cap\mcal{S}'))$
for all $i$ generates $R(X,\m(\mcal{S}'))$ and the claim follows.
\end{proof}

\begin{de}
Let $\mcal{S}$ be a monoid and let $f\colon\mcal{S}\rightarrow G$ be a superadditive map to a monoid $G$.

For every $s\in\mcal{S}$, the smallest positive integer $\iota_s$, if it exists, such that $f(\N\iota_ss)$ is an additive system is called the
{\em index\/} of $s$ (otherwise we set $\iota_s=\infty$).
\end{de}

I can finally make a connection with superlinear functions.

\begin{lem}\label{concave}
Let $X$ be a variety, $\mcal{S}$ a finitely generated monoid and let $f\colon\mcal{S}\rightarrow G$ be
a superadditive map to a monoid $G$ which is a subset
of $\Div(X)$ or $\bDiv(X)$, such that for every $s\in\mcal{S}$ the index $\iota_s$ is finite.

Then there is a unique superlinear function $f^\sharp\colon\mcal{S}_\R\rightarrow G_\R$ such that
for every $s\in\mcal{S}$ there is a positive integer $\lambda_s$ such that $f(\lambda_s s)=f^\sharp(\lambda_s s)$.
Furthermore, let $\mcal{C}$ be a rational polyhedral subcone of $\mcal{S}_\R$. Then $f_{|\mcal{C}\cap\mcal{S}}$ is additive up to
truncation if and only if $f^\sharp_{|\mcal{C}}$ is linear.
\end{lem}
\begin{proof}
The construction will show that $f^\sharp$ is the unique function with the stated properties.
To start with, fix a point $s\in\mcal{S}_\Q$ and let $\kappa$ be a positive integer such that $\kappa s\in\mcal{S}$. Set
$$f^\sharp(s):=\frac{f(\iota_{\kappa s}\kappa s)}{\iota_{\kappa s}\kappa}.$$
This is well-defined: take another $\kappa'$ such that $\kappa's\in\mcal{S}$. Then by the definition of the index we have
$$f(\iota_{\kappa s}\iota_{\kappa's}\kappa\kappa's)=
\iota_{\kappa s}\kappa f(\iota_{\kappa's}\kappa's)=\iota_{\kappa's}\kappa' f(\iota_{\kappa s}\kappa s),$$
so $f(\iota_{\kappa s}\kappa s)/\iota_{\kappa s}\kappa=f(\iota_{\kappa's}\kappa's)/\iota_{\kappa's}\kappa'$.

Now let $s\in\mcal{S}_\Q$, let $\xi$ be a positive rational number and let $\lambda$ be a sufficiently divisible positive integer such that
$\lambda\xi s\in\mcal{S}$. Then
$$f^\sharp(\xi s)=\frac{f\big((\iota_{\lambda\xi s}\lambda)\xi s\big)}{\iota_{\lambda\xi s}\lambda}=
\xi\frac{f\big((\iota_{\lambda\xi s}\lambda\xi)s\big)}{\iota_{\lambda\xi s}\lambda\xi}=\xi f^\sharp(s),$$
so $f^\sharp$ is positively homogeneous (with respect to rational scalars).
It is also superadditive: let $s_1,s_2\in\mcal{S}_\Q$ and let $\kappa$ be a sufficiently divisible positive integer such that
$f(\kappa s_1)=f^\sharp(\kappa s_1)$, $f(\kappa s_2)=f^\sharp(\kappa s_2)$ and
$f\big(\kappa(s_1+s_2)\big)=f^\sharp\big(\kappa(s_1+s_2)\big)$. By superadditivity of $f$ we have
$$f(\kappa s_1)+f(\kappa s_2)\leq f\big(\kappa(s_1+s_2)\big),$$
so dividing the inequality by $\kappa$ we obtain superadditivity of $f^\sharp$.

Let $E$ be any divisor on $X$, respectively any geometric valuation $E$ over $X$, when $G\subset\Div(X)$, respectively $G\subset\bDiv(X)$.
Consider the function $f^\sharp_E$ given by $f^\sharp_E(s)=\mult_Ef^\sharp(s)$.
Proposition \ref{Lip} applied to each $f^\sharp_E$ shows that $f^\sharp$ extends to a superlinear function on
the whole $\mcal{S}_\R$.

For the statement on cones, necessity is clear. Assume $f^\sharp|_{\mcal{C}}$ is linear and by Lemma \ref{gordan} let $e_1,\dots,e_n$ be
generators of $\mcal{C}\cap\mcal{S}$. For $s_0=e_1+\dots+e_n$ we have
\begin{equation}\label{eq9}
f^\sharp(s_0)=f^\sharp(e_1)+\dots+f^\sharp(e_n).
\end{equation}
Let $\mu$ be a positive integer such that $f(\mu s_0)=f^\sharp(\mu s_0)$ and $f(\mu e_i)=f^\sharp(\mu e_i)$ for all $i$.
From \eqref{eq9} we obtain
$$f(\mu s_0)=f(\mu e_1)+\dots+f(\mu e_n),$$
and Lemma \ref{linear} implies that $f^\sharp$ is additive on the truncation $\widehat{\mcal{S}}=\sum\N\mu e_i$ of $\mcal{C}\cap\mcal{S}$.
\end{proof}

\begin{de}
In the context of Lemma \ref{concave}, the function $f^\sharp$ is called {\em the straightening of $f$\/}.
\end{de}

\begin{re}\label{equal}
In the context of the assumptions of Lemma \ref{concave} let $s\in\mcal{S}$.
Let $\lambda$ be a positive integer such that $f^\sharp(\lambda s)=f(\lambda s)$. Then for every positive integer $\mu$ we have
$$f(\mu\lambda s)\geq\mu f(\lambda s)=\mu f^\sharp(\lambda s)=f^\sharp(\mu\lambda s)\geq f(\mu\lambda s),$$
so $f(\mu\lambda s)=\mu f(\lambda s)$. Therefore the index $\iota_s$ is the smallest integer $\lambda$ such that $f^\sharp(\lambda s)=f(\lambda s)$.
\end{re}

\begin{re}\label{remark}
The formulations of Conjectures A and B are in general the best possible, that is we cannot extend the results to the boundary of the cone
$\mcal S_\R$. For let $X$ be a variety, let $\mcal{S}=\N^2$ and assume $\m\colon\mcal{S}\rightarrow\bMob(X)$
is a superadditive map such that the system $\m(\mcal{S})$ is bounded and $\F$-saturated.
Let $\n\colon\mcal{S}\rightarrow\bMob(X)$ be the superadditive map given by
$$\n(s)=\begin{cases}
\m(s), & s\in\mcal S_\R\backslash\Int\mcal S_\R,\\
\m(2s), & s\in\Int\mcal S_\R.
\end{cases}$$
Since saturation is the property of rays by Lemma \ref{satrays}, the system $\n(\mcal{S})$ is again $\F$-saturated. However the algebra
$R(X,\n(\mcal{S}))$ is not finitely generated since the map $\n^\sharp$ is not continuous on the whole $\mcal S_\R$.
\end{re}

\section{Curve case}

In this section I will confirm Conjectures A and B on an affine curve.

\begin{te}\label{corollary}
Let $X$ be an affine curve, let $\mcal{S}$ be a finitely generated submonoid of $\N^r$ and
let $\m\colon\mcal{S}\rightarrow\Mob(X)$ be a superadditive map such that the system $\m(\mcal{S})$ is bounded and $\F$-saturated.
Let $\mcal{C}$ be a rational polyhedral cone in $\Int\mcal S_\R$.

Then the algebra $R(X,\m(\mcal{C}\cap\mcal{S}))$ is finitely generated.
\end{te}

\begin{re}\label{remark1}
Observe that on a curve b-divisors are just the usual divisors. Also all divisors move in the corresponding
linear systems, so the saturation condition reads
$$\lceil(\mu/\nu)\m(\nu s)+\F\rceil\leq \m(\mu s)$$
for every $s\in\mcal{S}$ and all positive integers $\mu$ and $\nu$.
By boundedness the limit $\lim\limits_{\mu\rightarrow\infty}\m(\mu s)/\mu$ exists for every $s\in\mcal{S}$.
Therefore for each $s\in\mcal{S}$ the algebra
$R(X,\m(\N s))$ is finitely generated, see \cite[2.3.10]{Cor07}, and thus the index of every $s\in\mcal{S}$ is finite.
Furthermore the map $\m^\sharp|_{\mcal{C}}$ is $\Q$-PL if and only if for every prime divisor $E$ in the support of
$\m(\mcal{S})$ the function $\m^\sharp_E|_{\mcal{C}}$ is $\Q$-PL, see the proof of Lemma \ref{concave}.
Also the saturation condition on a curve is a component-wise
condition, so from now on I assume the system $\m(\mcal{S})$ is supported at a point.
\end{re}

\begin{lem}\label{difference}
Let $X$ be an affine curve, let $\mcal{S}$ be a finitely generated monoid and let $\m\colon\mcal{S}\rightarrow\Mob(X)$ be a
superadditive map such that the system $\m(\mcal{S})$ is bounded, supported at a point $P$ and $\F$-saturated. Let
$\m^\sharp\colon\mcal S_\R\rightarrow\Mob(X)_{\R}$ be the straightening of $\m$.

Then there exists a constant $0<b\leq1/2$ with the following property: for each $s\in\mcal{S}$ either $\m^\sharp(s)=\m(s)$ or
$\m^\sharp(s)=\m(s)+e_sP$ for some $e_s$ with $b\leq e_s\leq 1-b$.
\end{lem}
\begin{proof}
Let $\F=-fP$ with $f<1$. Fix $s\in\mcal{S}$ and assume $\m^\sharp(s)\neq \m(s)$. Then there is the smallest positive integer $\lambda$ such that
$\m\big((\lambda+1)s\big)\neq(\lambda+1)\m(s)$; in particular
$$\m(\lambda s)=\lambda \m(s)$$
and
$$\m\big((\lambda+1)s\big)=(\lambda+1)\m(s)+e_{\lambda s}P$$
for some $e_{\lambda s}\geq1$. From the saturation condition we have
$$\big\lceil\big(\lambda/(\lambda+1)\big)\m\big((\lambda+1)s\big)-fP\big\rceil\leq \m(\lambda s),$$
that is
$$\big\lceil\m(\lambda s)+\big(\lambda/(\lambda+1)\big)e_{\lambda s}P-fP\big\rceil\leq \m(\lambda s).$$
This implies $\lambda/(\lambda+1)\leq f$, and so $1/(\lambda+1)\geq1-f$. Therefore
$$\m^\sharp(s)\geq\frac{1}{\lambda+1}\m\big((\lambda+1)s\big)=\m(s)+\frac{1}{\lambda+1}e_{\lambda s}P\geq\m(s)+(1-f)P.$$

On the other hand, let $\kappa$ be a positive integer such that $\m^\sharp(\kappa s)=\m(\kappa s)$. Then saturation gives
$$\lceil(1/\kappa)\m(\kappa s)-fP\rceil\leq \m(s),$$
that is
$$\lceil\m^\sharp(s)-fP\rceil\leq \m(s).$$
Hence
$$\m^\sharp(s)-\m(s)\leq fP.$$
In particular if $f\leq1/2$ then $\m^\sharp(s)=\m(s)$ for every $s\in\mcal{S}$. Set $b:=\min\{1-f,1/2\}$.
\end{proof}

\begin{lem}\label{index}
Let $X$ be an affine curve, let $\mcal{S}$ be a finitely generated monoid and let $\m\colon\mcal{S}\rightarrow\Mob(X)$ be a
superadditive map such that the system $\m(\mcal{S})$ is bounded, supported at a point $P$ and $\F$-saturated. Let $b$ be the constant
from Lemma \ref{difference}. Then for each $s\in\mcal{S}$ we have $\iota_s\leq1/b$.
\end{lem}
\begin{proof}
Let $\m^\sharp\colon\mcal S_\R\rightarrow\Mob(X)_{\R}$ be the straightening of $\m$.
Observe that Lemma \ref{difference} implies that $\m(s)=\lfloor\m^\sharp(s)\rfloor$ for each $s\in\mcal{S}$, and this in turn implies that
the index $\iota_s$ is the smallest integer $\lambda$ such that $\m^\sharp(\lambda s)$ is an integral divisor (cf.\ Remark \ref{equal}).

Now fix $s\in\mcal{S}$, assume $\iota_s>1$ and let $\m^\sharp(\iota_ss)=\m(\iota_ss)=\mu_sP$.
Notice that $\iota_s$ and $\mu_s$ must be coprime: otherwise assume $p$ is a prime dividing both $\iota_s$ and $\mu_s$. Then
$$\m^\sharp\big((\iota_s/p)s\big)=(\mu_s/p)P$$
is an integral divisor and so $\iota_s$ is not the index of $s$, a contradiction. Therefore there is an integer
$1\leq\kappa\leq\iota_s-1$ such that $\kappa\mu_s\equiv1\pmod{\iota_s}$, and therefore
$$\m^\sharp(\kappa s)=(\kappa\mu_s/\iota_s)P\quad\textrm{and}\quad\m(\kappa s)=\big((\kappa\mu_s-1)/\iota_s\big)P.$$
Combining this with Lemma \ref{difference} we obtain
$$bP\leq\m^\sharp(\kappa s)-\m(\kappa s)=(1/\iota_s)P,$$
and finally $\iota_s\leq1/b$.
\end{proof}

Finally we have

\begin{proof}[Proof of Theorem \ref{corollary}]
By Lemma \ref{gordan} the monoid $\mcal{S}'=\mcal{C}\cap\mcal{S}$ is finitely generated and let $e_1,\dots,e_n$ be its generators.
We have $\mcal S'_\R=\mcal{C}$ and $\m^\sharp$ is continuous on $\mcal S'_\R$.
Setting $\kappa:=\lfloor1/b\rfloor!$ for $b$ as in Lemma \ref{difference},
and taking the truncation $\widehat{\mcal{S}}=\sum_{i=1}^n\N\kappa e_i$ of $\mcal{S}'$ we have
that $\m^\sharp(s)=\m(s)$ for every $s\in\widehat{\mcal{S}}$ by Lemma \ref{index} and $\mcal S'_\R=\widehat{\mcal{S}}_\R$.
By Remark \ref{remark1} I assume the system $\m(\mcal{S})$ is supported at a point.

By Corollary \ref{PL} applied to the monoid $\widehat{\mcal{S}}$ the map $\m^\sharp|_{\widehat{\mcal{S}}_\R}$ is $\Q$-PL
and thus the algebra $R(X,\m(\mcal{S}'))$
is finitely generated by Lemmas \ref{limit} and \ref{concave}.
\end{proof}

\bibliography{biblio}
\pagestyle{plain}
\end{document}